\documentclass[usenames,dvipsnames]{amsart}

\usepackage{amsmath,amssymb,amsthm,graphicx}
\usepackage[subrefformat=parens,labelformat=parens]{subfig}
\usepackage[all]{xy}   
\usepackage{hyperref}
\usepackage{verbatim}
\usepackage[dvipsnames]{xcolor}
\usepackage{tikz}
\usetikzlibrary{positioning,shapes,shadows}
\usetikzlibrary{arrows,automata}
\usepackage[shortlabels]{enumitem}
\usepackage{mathrsfs}

\newtheorem{theorem}{Theorem}[section]
\newtheorem{proposition}[theorem]{Proposition}

\newtheorem*{conjecture}{Boone--Higman Conjecture}

\theoremstyle{definition}

\newtheorem*{definition*}{Definition}

\newtheorem{notation}[theorem]{Notation}

\newtheorem{example}[theorem]{Example}

\renewcommand{\star}{\text{\small$*$}}

\newcommand{\N}{\mathbb{N}}
\newcommand{\R}{\mathscr{R}}
\newcommand{\Nuc}{\mathcal{N}}

\newcommand{\C}{\mathfrak{C}}

\tikzset{initial text={}}
\makeatletter
\tikzstyle{initial by arrow}=   [after node path=
{
  {
    [to path=
    {
      [->,double=none,every initial by arrow]
      ([shift=(\tikz@initial@angle:\tikz@initial@distance)]\tikztostart.\tikz@initial@angle)
          node [shape=rectangle,anchor=\tikz@initial@anchor] {\tikz@initial@text}
        -- (\tikztostart)}]
    edge ()
  }
}]
\makeatother

\begin{document}

\title[Hyperbolic into simple]{Embedding hyperbolic groups into finitely presented infinite simple groups}

\author[J.~Belk]{James Belk}
\address{School of Mathematics and Statistics, University of Glasgow, University Place, Glasgow, G128QQ, Scotland.}
\email{\href{mailto:jim.belk@glasgow.ac.uk
}{jim.belk@glasgow.ac.uk
}}

\author[C.~Bleak]{Collin Bleak}
\address{School of Mathematics and Statistics, University of St Andrews, St Andrews, Scotland KY8 6ER.}
\email{\href{mailto:collin.bleak@st-andrews.ac.uk}{collin.bleak@st-andrews.ac.uk}}

\maketitle

\begin{abstract} The Boone--Higman conjecture is that every recursively presented group with solvable word problem embeds in a finitely presented simple group.  We discuss a brief history of this conjecture and work towards it.  Along the way we describe some classes of finitely presented simple groups, and we briefly outline work of Belk, Bleak, Matucci, and Zaremsky showing that the broad class of hyperbolic groups embeds in a class of finitely presented simple groups.
\end{abstract}
 \section{Boone--Higman Conjecture}
 In this brief note, we trace through some history and constructions (both old and new) related to the Boone--Higman conjecture that every recursively presented group with solvable word problem embeds in a finitely presented simple group.

 We note that progress towards the conjecture has generally been hindered by our lack of knowledge (as mathematicians) of broad enough families of finitely presented simple groups to serve as hosts to Boone and Higman's conjectured embeddings, an issue that now seems to be easing.
 \subsection{The word problem for finitely generated groups.}

 In 1911 Max Dehn proposed several problems where progress would be needed before a meaningful theory of groups given by presentations could be developed~\cite{Dehn}.  One of these problems is now known as the \textit{word problem}.

 \begin{notation}[Word Problem] Given a finitely presented group
\[G = \langle x_1,x_2,\ldots,x_n\mid r_1,r_2,\ldots,r_m\rangle\]
give a procedure by which one can determine, for any given product expression involving the generators (and their inverses), whether or not the expression represents the identity of $G$.
\end{notation}

It turns out that the word problem is unsolvable in general.  That is, there exist finitely presented groups for which there is no algorithm to decide whether a given word in the generators represents the identity.  The first examples of such groups were given by Pyotr Novikov in 1955 \cite{Novikov55} and independently by William Boone in 1959 \cite{Boone59}.

 \subsection{Key Word Problem Results}
 There are several key results in the area of the word problem for groups.  We describe some of these here with some discussion.
 
 Firstly, and as mentioned above, the word problem cannot always be solved.
 \begin{theorem}[Novikov 1955 \cite{Novikov55} $\mid$ Boone 1959 \cite{Boone59}]
 There exists a finitely presented group $G$ with non-solvable word problem. 
 \end{theorem}
 A second important insight is that the theory of groups with solvable word problem is linked with the theory of simple groups.
\begin{theorem}[Kuznetsov 1958 \cite{Kuznetsov}]
A finitely presented simple group G has solvable word problem.
\end{theorem}
\begin{proof}
Suppose 
\[
G = \langle x_1,x_2,\ldots,x_n\mid r_1,r_2,\ldots,r_m\rangle
\]
is a finite presentation of a simple group $G$, and $w\in \bigl(X\cup X^{-1}\bigr)^{\star}$ is a word in the generators and their inverses.

Set $H$ to be the group presented as follows:
\[
H=\langle X\mid r_1,r_2,\ldots,r_m,w\rangle.
\]
Since $G$ is simple, the group $H$ is nontrivial if and only if $w$ represents the identity in~$G$. Now run two algorithms in parallel: Enumerate all consequences of relators in $G$ and $H$.  If $w$ is a consequence of the relators in $G$ then it represents the identity in $G$.  If all the generators in $X$ are consequences of the relations of $H$, then $w$ is not the identity in $G$.  One of these algorithms must terminate positively.
\end{proof}

And finally, the core result of Boone and Higman.
 
 \begin{theorem}[Boone--Higman 1974 \cite{BooneHigman74}]
 A finitely generated group $G$ has solvable word problem  if and only if it embeds in a simple subgroup of a finitely presented group.
 \end{theorem} 

 It is an interesting side story that it seems that Boone and Higman were not aware of Kuznetsov's Theorem until a comment of Richard Thompson at a conference in Irvine in 1969, which seems to have opened the door towards their work (this is discussed in \cite{BooneHigman74}).  In any case, Boone and Higman had already made their conjecture by 1973 (see \cite{Boone73}).
\subsection{The conjecture and a new result towards it}
We are now in position to state the conjecture.
\begin{conjecture}[1973]
A finitely generated group has solvable word problem if and only if it embeds in a finitely presented simple group.
\end{conjecture}

In the remainder of this note, we trace through the key ideas of a new theorem \cite{BelkBleakMatucciZaremsky} which confirms the Boone--Higman conjecture for a natural first ``broad'' class of groups.
\begin{theorem}[Belk--Bleak--Matucci--Zaremsky]
Every hyperbolic group embeds in a finitely presented simple group.
\end{theorem}

 \section{Hyperbolic Groups}
 Motivated by a construction and algorithm of Dehn, a class of groups with a very fast (linear time) algorithm for solving the word problem was defined.  As we will describe in Subsection \ref{subsec:Gromov}, this class eventually became known as the  \textit{hyperbolic groups}.

 \subsection{Surface groups and Dehn's algorithm}
 Max Dehn in 1911 proposed the word problem as being one of three fundamental problems for finitely presented groups.  He also gave an elegant solution to the word problem for $G_S=\pi_1(S,*)$ for $S$ a closed orientable surface of genus at least $2$.  
 
 \begin{center}
 \begin{minipage}{0.6\textwidth}
     \includegraphics[width=2.8in, height=1 in]{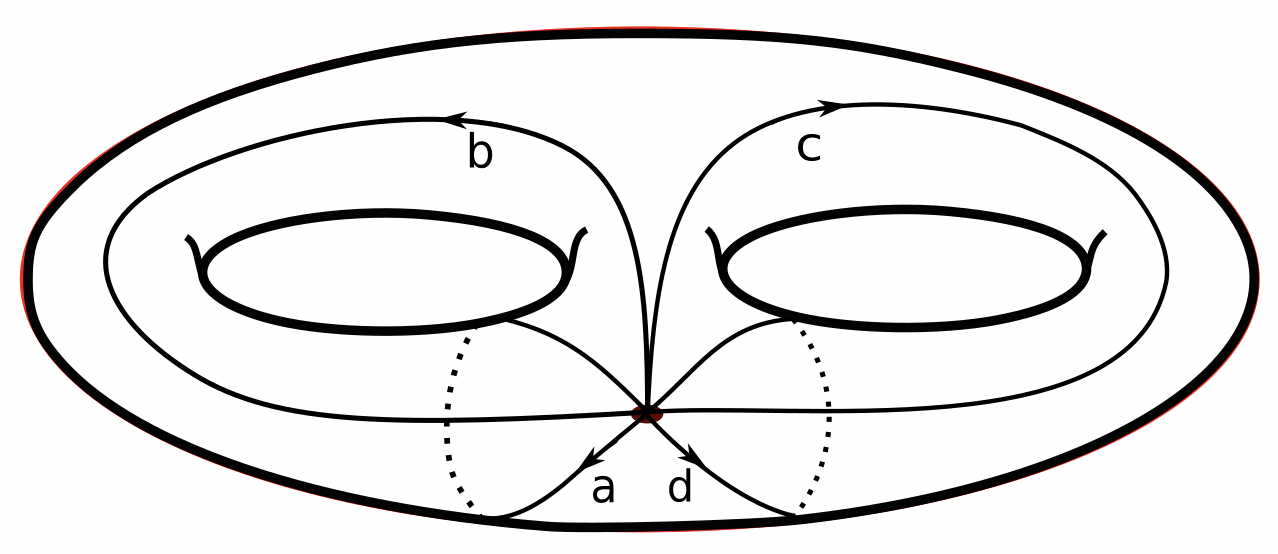}\\
     \begin{center}$G_S=\langle a,b,c,d\mid aba^{-1}b^{-1}cdc^{-1}d^{-1}\rangle$     
\end{center} \end{minipage}
 \begin{minipage}{0.39\textwidth}
\vspace{-.2 in}
\hspace{.2 in}
    \includegraphics[width=1.4in, height=1.3 in]{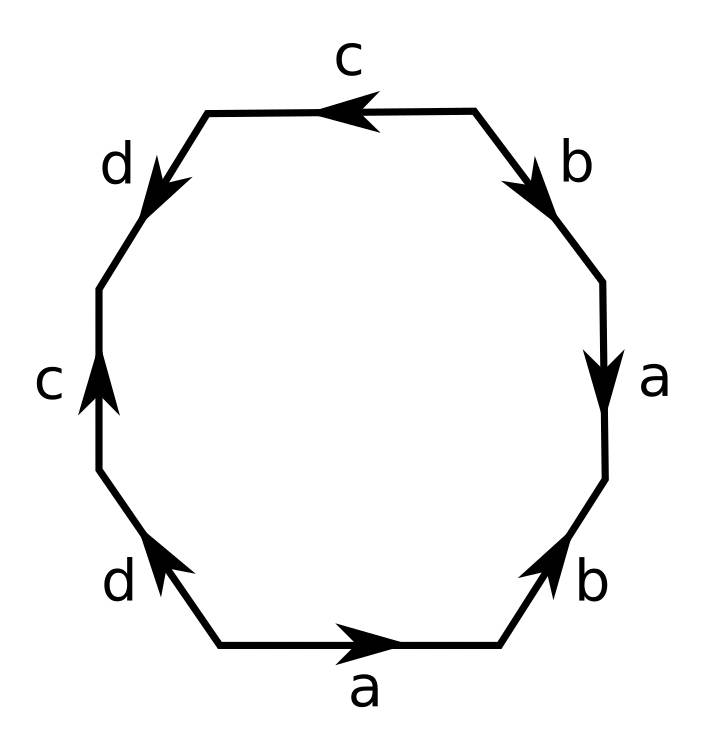}\\
 \end{minipage}
 \end{center}

 When the genus of $S$ is greater than one there is a presentation of $G_S$ with ``small cancellation properties'' which enable a very fast algorithm to solve the word problem for the group.  We call such a presentation a \textit{Dehn presentation}.

  The key ingredient for a Dehn presentation is that each basic relator of the presentation represents two circular relators (for the example presentation of $G_S$ above, reading around the octagon from any start position in the two basic directions). 
 Then, the overall set of created circular relators must not share any long overlapping common word: any maximal common substrings of two such circular relators must have length less than half the length of each of the circular relators involved.  For the presentation of $G_S$ above, any two strings read off two different ways from the octagon can never share a common substring longer than one letter, and $1/8<1/2$.

\subsection{Dehn presentations and hyperbolic groups}\label{subsec:Gromov}
 The class of groups which admit a Dehn presentation became a well-studied class of groups over time.  Inspired from various existing mathematical theories including hyperbolic geometry and low-dimensional topology (including Max Dehn's results for surface groups) Gromov  found a beautiful geometric interpretation of these groups \cite{Gromov}.  Gromov understood that these are precisely those groups which act ``geometrically'' (with a properly discontinuous and cocompact action by isometries) on hyperbolic spaces, and called the resulting class of groups \emph{hyperbolic groups}.
 
 

 Suppose \[G=\langle X\mid R\rangle\] is a Dehn presentation for $G$.  As from this presentation, no two induced circular relators overlap on a long substring of either relator, we have that any  $w\in (X\cup X^{-1})^*$  with $w =_G 1_G$ must contain some long substring which is from one of the relators, or, a freely cancelling pair of letters. In the first case one replaces the found partial relator by the inverse of its complementary part, which is a shorter word, without changing the element the string represents.  In the second case, one carries out a free cancellation, also reducing the length of the word.  Now simply repeat such reductions until the whole word collapses: this is Dehn's algorithm.
 
By way of example, below, we follow Dehn's algorithm for a short word equivalent to the identity from the presentation of $G_S$. 
  \begin{example}
  [Dehn's algorithm by example]
      \[G_S=\langle a,b,c,d\mid aba^{-1}b^{-1}cdc^{-1}d^{-1}\rangle\]

\vspace {-.2 in}
      \begin{align*}
      d^{-1}acdc^{-1}d^{-1}aba^{-1}{\color{red}{a^{-1}b^{-1}cdc^{-1}}}\\
      d^{-1}acdc^{-1}d^{-1}aba^{-1}{\color{blue}{b^{-1}a^{-1}d}}\\
      d^{-1}a{\color{red}cdc^{-1}d^{-1}aba^{-1}b^{-1}} a^{-1}d\\
      d^{-1}{\color{red}a a^{-1}}d\\
    {\color{red}d^{-1} d}\\
       \varepsilon\\
      \end{align*}
\end{example}
  \begin{theorem}[Gromov 1987, Ollivier 2005]
With overwhelming probability, a random group in the few relator model with various lengths is hyperbolic.\end{theorem}
     
Thus, hyperbolic groups form a natural first ``port-of-call'' for verifying the Boone--Higman conjecture: a vast class of groups essentially defined by the property of having a nice solution to the word problem.
     

\section{Homeomorphism groups of Cantor space}

\subsection{Cantor space $\C$}

A \textit{Cantor space} is any space which is homeomorphic to the usual middle-thirds Cantor set. By a theorem of Brouwer, any compact, totally disconnected metrizable space without isolated points is a Cantor space.

A basic example of a Cantor space is the infinite product $\C=\{0,1\}^\omega$, whose points are infinite binary sequences.  Given a finite binary sequence $\alpha$, the corresponding \textit{cone} $C_\alpha$ is the set of all points that have $\alpha$ as a prefix.  Under the homeomorphism from $\C$ to the middle-thirds Cantor set, these map to subsets which are similar to the whole Cantor set:
\begin{center}
\includegraphics{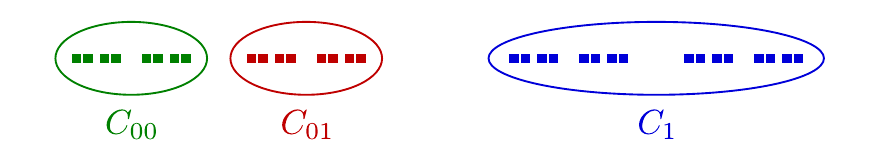}
\end{center}
Note that the \textit{clopen sets} in $\C$ (i.e.\ the sets that are both closed and open) are precisely the sets that are finite unions of cones.

Any two cones in $\C$ have a \emph{canonical homeomorphism} between them.
\begin{center}
\includegraphics{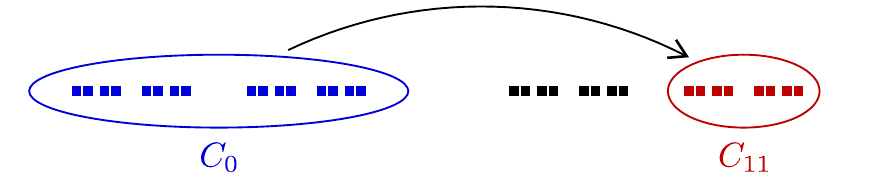}
\end{center}
Specifically, the canonical homeomorphism from $C_\alpha$ to $C_\beta$ maps each sequence in $C_\alpha$ to the sequence in $C_\beta$ obtained by replacing the prefix $\alpha$ with the prefix~$\beta$.

\subsection{The rational group $\mathscr{R}$}

If $f\colon \C\to \C$ is a homeomorphism and $C_\alpha$ is a cone in the domain, the \textit{local action} of $f$ on $C_\alpha$ is the map $f_\alpha\colon \C\to \C$ that fits into a commutative square
\[
{
\xymatrix@C=0.35in@R=0.35in{
  C_\alpha\ar^{f}[r]\ar_{\sim}[d] &   C_{\beta} \ar^{\sim}[d] \\
  \C\ar_{f_\alpha}[r] &   \C
}
}
\]
where the vertical maps are canonical homeomorphisms, and $C_\beta$ is the smallest cone that contains $f(C_\alpha)$.  A homeomorphism $f$ is \textit{rational} if it has only finitely many different local actions.

In 2000, Grigorchuk, Nekrashevych, and Suschanski\u\i\ observed that the rational homeomorphisms of $\C$ form a group under composition, which they refer to as the \textit{rational group}~$\mathscr{R}$ (see \cite{GNS2000}).  Though we have defined $\mathscr{R}$ using the binary alphabet $\{0,1\}$, they showed that the group $\mathscr{R}_d$ determined by any finite alphabet with $d\geq 2$ letters is isomorphic to the binary group.  They also showed that rational homeomorphisms could be described in a certain way using finite-state transducers, with one state for each of the local actions, as in the example below.
\begin{center}
\boldmath
\begin{minipage}{0.49\textwidth}
\begin{center}
\begin{tikzpicture}[shorten >=0.5pt,node distance=2cm,on grid,auto, semithick,every state/.style={fill=blue!25!white,text=black}]
\node[rectangle,draw=black!150,fill=blue!25!white,text=black] (a) at (90:1.2) {$a$};
\node[circle,draw=black!50,fill=blue!25!white,text=black] (b) at (270:1.2) {$b$};
\path [->]
(a) edge [->]node{\!\small$0/\varepsilon$}(b)
(a) edge [min distance = 0.9 cm, in = 105,out=75]node[swap]{$1/11$}(a)
(b) edge [->,out=135,in=225,looseness=1.5]node{\small$0/0$}(a)
(b) edge [->,,out=45,in=315,looseness=1.5]node[swap]{\small$1/10$}(a);
\end{tikzpicture}
\end{center}
\end{minipage}
\begin{minipage}{0.49\textwidth}

\vspace{.5 in}
\hspace{.05 in} On infinite rooted binary tree,

\vspace{.2 in}
\begin{tikzpicture}[
      inner sep=0pt,
      baseline=-30pt,
      level distance=20pt,
      level 1/.style={sibling distance=30pt},
      level 2/.style={sibling distance=15pt},
      level 3/.style={sibling distance=7pt}
    ]
    \node (root) [circle,fill] {}
    child {node (0) [circle,fill] {}
      child {node (00) {$00$}}
      child {node (01) {$01$}}}
    child {node (1) {$1$}};
  \end{tikzpicture}
  $\;\longrightarrow\;$
  \begin{tikzpicture}[
      inner sep=0pt,
      baseline=-30pt,
      level distance=20pt,
      level 1/.style={sibling distance=30pt},
      level 2/.style={sibling distance=15pt},
      level 3/.style={sibling distance=7pt}
    ]
    \node (root) [circle,fill] {}
    child {node (0) {$0$}}
    child {node (1) [circle,fill] {}
      child {node (10) {$10$}}
      child {node (11) {$11$}}};
  \end{tikzpicture}
  
  \vspace{.25 in}
  \hspace{.5 in}
  iteratively.\\ 
  \end{minipage}
\end{center}
Belk, Hyde, and Matucci later proved that the group $\mathscr{R}$ is simple, but not finitely generated~\cite{Belk-Hyde-Matucci-1}.  

A rational homeomorphism is \textit{synchronous} if it maps each cone in the domain to a cone in the range (so each local action $f_\alpha$ is a homeomorphism).  Synchronous rational homeomorphisms in $\mathscr{R}_d$ act as automorphisms of the infinite $d$-ary tree, of which the $d$-ary Cantor space $\{0,\ldots,d-1\}^\omega$ is the boundary.  Groups of synchronous rational homeomorphisms include the famous Grigorchuk group of intermediate growth, as well as other self-similar groups.  On the other hand, Thompson's groups $F$, $T$, and $V$ act on $\C$ by homeomorphisms which are rational but not necessarily synchronous.

\subsection{Contracting groups}

A group $G\leq \mathscr{R}_d$ of synchronous rational homeomorphisms is \textit{self-similar} if every local action of every element of $G$ again lies in~$G$.  The theory of self-similar groups has been developed extensively by Nekrashevych~\cite{Nekrashevych-book}, and includes groups such as the Grigorchuk group, the Gupta--Sidki groups, and iterated monodromy groups for holomorphic functions.

The most important class of self-similar groups are the contracting groups.  Here a group $G\leq\mathscr{R}_d$ is \textit{contracting} if the union $\mathcal{N}$ of the cores of its elements is finite, where the \textit{core} of a rational homeomorphism $f$ is the set of local actions of $f$ that occur on infinitely many different cones.  In this case, the finite set $\mathcal{N}$ is called the \textit{nucleus} of~$G$.  Nekrashevych has developed notions of limit spaces and partial self-coverings for contracting self-similar groups, revealing deep connections between these groups and the theory of dynamical systems.

If $G\leq\mathscr{R}_d$ is a self-similar group, the corresponding \textit{R\"over--Nekrashevych group} $V_dG$ is the subgroup of $\mathscr{R}_d$ generated by $G$ and the $d$-ary Thompson group~$V_d$.  Such groups were investigated by Scott \cite{Scott1,Scott2,Scott3}, and then R\"over \cite{roever}, and the connection to self-similar groups was made by Nekrashevych~\cite{NekraCuntz}.

\begin{theorem}[Nekrashevych 2018 \cite{NekraFP}]If $G\leq\mathscr{R}_d$ is a contracting self-similar group, then $V_dG$ is finitely presented.
\end{theorem}

\section{Hyperbolic in Rational}

In 2021, the authors and Francesco Matucci proved the following theorem.

\begin{theorem}[Belk--Bleak--Matucci 2021 \cite{BelkBleakMatucci}] Every hyperbolic group $G$ embeds in the rational group $\mathscr{R}$. \end{theorem}

The key idea is to consider the action of $G$ on a certain totally disconnected boundary $\partial_h G$ defined by Gromov, known as the \textit{horofunction boundary}.  In the case where this boundary is a Cantor space and the action is faithful, we construct a homeomorphism $\partial_h G\to \{0,1\}^\omega$ that conjugates $G$ into~$\mathscr{R}$.  (If $\partial_hG$ is not a Cantor space or the action is not faithful, this can be rectified by replacing $G$ with the free product $G*F_2$.)

The construction of the homeomorphism $\partial_h G\to \{0,1\}^\omega$ involves realizing $\partial_h G$ as the boundary of a certain tree, which we call the \textit{tree of atoms}.  For a hyperbolic group $G$, this tree has the structure of a \textit{self-similar tree}, as defined below, which is fundamental to the proof of rationality.

\subsection{Self-Similar Trees}

The definition of a self-similar tree is technical, but the main idea is that a self-similar tree is an infinite, rooted tree with finitely many ``types'' of nodes, where nodes of the same type are related in a certain way.  Here is a typical example:
\begin{center}
\includegraphics{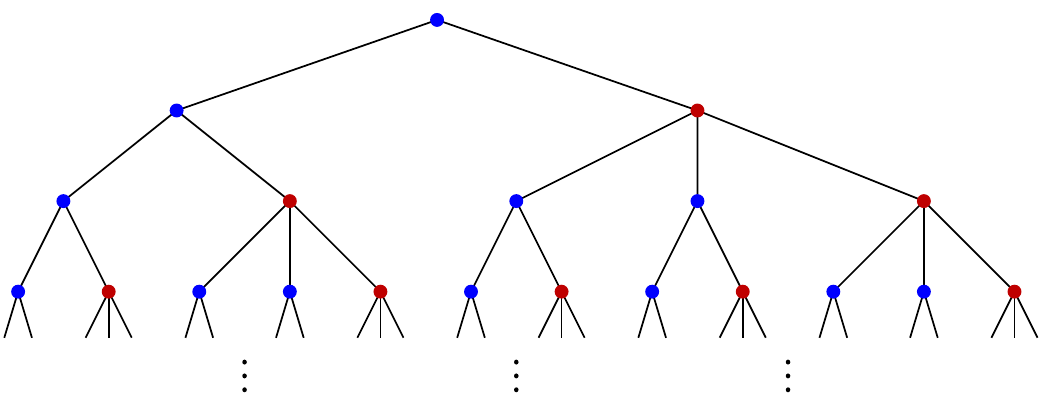}
\end{center}
In this case, there are two types of nodes (indicated by red and blue dots), where each blue node has one red and one blue child, and each red node has one red and two blue children.

In the picture above, the children of each node have a left-to-right order, but this is not necessarily part of the structure of a self-similar tree.  Instead, if two nodes have the same type, a self-similar tree specifies a finite set of isomorphisms between the subtrees rooted at those nodes.  These isomorphisms form a groupoid, and are closed under taking restrictions. 
 (This is the technical part of the definition.) 

If $T$ is a self-similar tree, its boundary $\partial T$ is a compact, totally disconnected metrizable space.  The notion of a ``rational homeomorphism'' can be extended to homeomorphisms of $\partial T$ by defining what it means for a homeomorphism to have finitely many local actions. 

\begin{theorem}[Belk--Bleak--Matucci 2021 \cite{BelkBleakMatucci}]
If the boundary $\partial T$ of a self-similar tree has no isolated points, then its group of rational homeomorphisms is isomorphic to~$\mathscr{R}$.
\end{theorem}

\subsection{The Tree of Atoms}

To prove that every hyperbolic group $G$ embeds into $\mathscr{R}$, we must show that $G$ acts faithfully by rational homeomorphisms on the boundary of some self-similar tree.  This tree is the tree of atoms, whose boundary is the horofunction boundary $\partial_h G$ of $G$.

The tree of atoms and horofunction boundary can be defined for any locally finite graph $X$.  To do so, fix a base vertex in $X$, and for each $n\geq 0$ let $B_n$ be the $n$-ball centered at the base vertex, including both vertices and edges.  Given any vertex $v$ in $X$, the corresponding \textit{vector field} on $B_n$ is obtained by orienting each edge of $B_n$ in the direction that points towards~$v$:
\begin{center}
\includegraphics{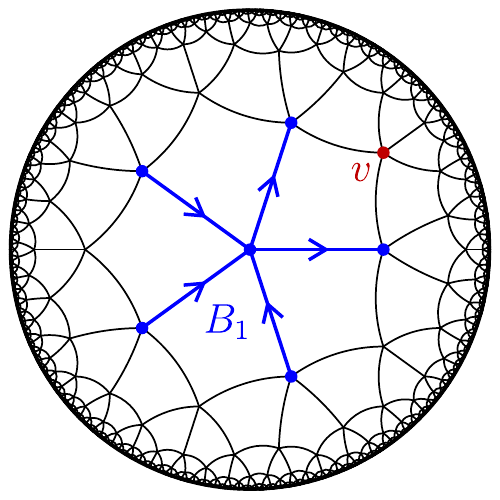}
\end{center}
That is, each edge of $B_n$ is oriented so that its terminal vertex is closer to $v$ than its initial vertex.  Edges whose two endpoints have the same distance from $v$ are not oriented.

Different vertices in $X$ correspond to different vector fields on $B_n$, but there are only finitely many possibilities.  The set of all vertices $v$ that determine a given vector field on $B_n$ is an \textit{$n$th-level atom}.  There are finitely many such atoms, and these form a partition of the vertices of~$X$.
\begin{center}
\includegraphics{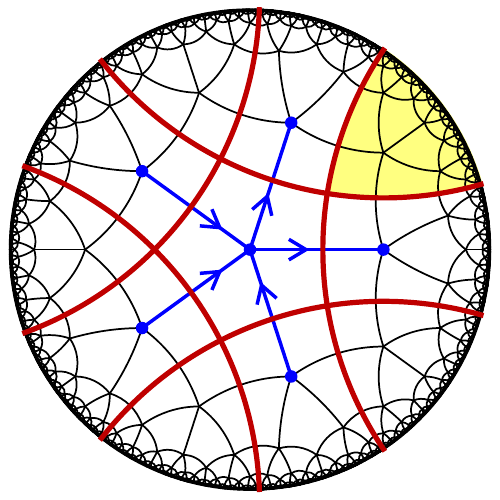}
\end{center}
In the example above, $X$ has been partitioned into eleven $1$st-level atoms, with the highlighted atom corresponding to the indicated vector field on~$B_1$.  Each of these atoms is subdivided further at level two.
\begin{center}
\includegraphics{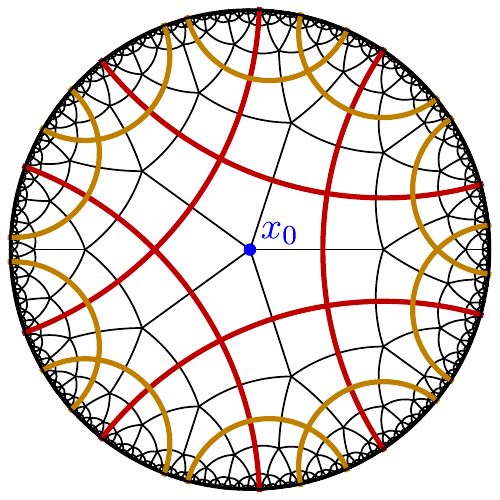}
\end{center}
The \textit{tree of atoms} for $X$ is the rooted tree of all infinite atoms in~$X$.  The root is the unique $0$th-level atom, which is the whole graph~$X$.  In the above example, this root has ten children, corresponding to the ten infinite atoms, each of these $1$st-level nodes has three children, and so forth.

The boundary of the tree of atoms is the \textit{horofunction boundary} $\partial_h X$ of~$X$, as defined by Gromov.  This space is totally disconnected and compact, and if $X$ is hyperbolic there is a finite-to-one surjection from $\partial_h X$ to the Gromov boundary~$\partial X$ (see \cite{perego2023rationality}).

\begin{theorem}[Belk--Bleak--Matucci 2021 \cite{BelkBleakMatucci}]If $G$ is a hyperbolic group acting geometrically on $X$, then the tree of atoms for $X$ has the structure of a self-similar tree, and elements of $G$ act as rational homeomorphisms of~$\partial_h X$.
\end{theorem}

In particular, if $X$ is a Cayley graph of $G$, then $G$ acts geometrically on~$X$, and we write $\partial_h G$ for the corresponding horofunction boundary (though $\partial_h G$ depends on the generating set chosen for the Cayley graph).  As long as $\partial_h G$ has no isolated points and $G$ acts faithfully on $\partial_h G$,  this theorem gives an embedding of $G$ into the rational group.  As mentioned above, these two conditions can always be assured by replacing $G$ with $G*F_2$.

\section{Hyperbolic in Simple}

In upcoming work, the authors together with Francesco Matucci and Matthew Zaremsky establish the Boone--Higman conjecture for hyperbolic groups.

\begin{theorem}[Belk--Bleak--Matucci--Zaremsky 2023 \cite{BelkBleakMatucciZaremsky}]
Every hyperbolic group $G$ embeds into a finitely presented simple group.
\end{theorem}

The key idea is to consider the group $[[G]]$ of all piecewise-$G$ homeomorphisms of the horofunction boundary~$\partial_h G$.  Here a homeomorphism $f$ of $\partial_h G$ is \textit{piecewise}-$G$ if there exists a finite partition of $\partial_h G$ into clopen sets such that $f$ agrees with an element of $G$ on each set from the partition.

Since $G$ acts by rational homeomorphisms of $\partial_h G$, the same holds for $[[G]]$.  The authors together with Matucci and Zaremsky prove that this group is contracting, and by a generalization of Nekrasevych's finite presentability theorem for R\"over--Nekrashevych groups it follows that $[[G]]$ is finitely presented.

Unfortunately, $[[G]]$ is not typically a simple group, though it does have simple commutator subgroup.  However, the first author and Matthew Zaremsky have defined a class of finitely presented simple groups called \textit{twisted Brin--Thompson groups} \cite{BelkZaremskyTwisted} into which the group $[[G]]$ can be shown to embed.

\subsection{Contracting Asynchronous Groups}

If $H$ is a group of homeomorphisms of a Cantor space $X$, we say that $H$ is \textit{full} if every piecewise-$H$ homeomorphism of $X$ is an element of~$H$.  This terminology comes from the theory of \'etale groupoids, where $H$ is full if and only if it is the topological full group of the corresponding groupoid of germs.  For example, the Thompson groups~$V_d$ are full, as is any R\"over--Nekrashevych group.  If $G$ is a hyperbolic group, then $[[G]]$ is a full group of homeomorphisms of the horofunction boundary~$\partial_h G$.

The authors together with Matucci and Zaremsky prove the following generalization of Nekrashevych's theorem on the finite presentability of R\"over--Nekrashevych groups.

\begin{theorem}[Belk--Bleak--Matucci--Zaremsky 2023 \cite{BelkBleakMatucciZaremsky}]\label{thm:FinitelyPresentedTheorem}
Let $T$ be a self-similar tree whose boundary $\partial T$ is a Cantor space, and let $H$ be a group of rational homeomorphisms of~$\partial T$.  If $V_T$ is minimal and $H$ is full, contracting, and contains $V_T$, then $H$ is finitely presented.
\end{theorem}

Here $V_T$ is a naturally defined Thompson-like group of homeomorphisms associated to $T$, consisting of all homeomorphisms that piecewise agree with the isomorphisms between subtrees that define the self-similar structure.  The group $V_T$ is \textit{minimal} if the orbit of every point in $\partial T$ is dense.  The requirement that $V_T$ is minimal ensures that $V_T$ is itself finitely presented.

In the case of a hyperbolic group $G$, we prove that the action of $G$ on its horofunction boundary $\partial_h G$ is contracting, and it follows that the same holds for the group $[[G]]$ of piecewise-$G$ homeomorphisms.  The group $[[G]]$ is full and contains $V_T$, where $T$ is the tree of atoms for~$G$.  Finally, we can ensure $V_T$ is minimal by replacing $G$ with $G*F_2$, if necessary. It follows that $[[G]]$ is finitely presented.
 
\subsection{Twisted Brin--Thompson groups} 

In 2022, the first author and Matthew Zaremsky introduced a new class of groups called twisted Brin--Thompson groups \cite{BelkZaremskyTwisted}, which are a variation on the higher-dimensional Thompson groups $nV$ defined by Brin \cite{BrinHigherV,BrinHigherVPresentations}. 

Specifically, if $S$ is any countable set, let $\C^S$ be the Cantor space $\prod_S \C$, where $\C=\{0,1\}^\omega$.  If $G$ is any group of permutations of $S$, then the restricted wreath product $W = V\wr G = \bigl(\bigoplus_{S} V)\rtimes G$ acts on $\C^S$ by homeomorphisms.   The \textit{twisted Brin--Thompson group} $SV_G$ is the group of all homeomorphisms of $\C^S$ which are piecewise-$W$. 

Belk and Zaremsky prove that $SV_G$ is always simple, and give conditions under which it is finitely presented \cite{BelkZaremskyTwisted}.  These conditions were later improved by Zaremsky, yielding the following theorem.

\begin{theorem}[Zaremsky 2022 \cite{zaremskyTaste}]
Let $G$ be a finitely presented, oligomorphic group of permutations of a set $S$. Suppose the stabilizer in $G$ of any finite subset of $S$ is finitely generated. Then
$G$ embeds as a subgroup of a finitely presented simple group, namely the twisted Brin–Thompson group $SV_G$.
\end{theorem}

Here $G$ is \textit{oligomorphic} if the induced action of $G$ on $S^k$ has finitely many orbits for every $k\geq 1$.  Oligomorphic groups were introduced by Peter Cameron~\cite{CameronOligomorphic}.

It is easy to see that any full group of homeomorphisms of a Cantor space acts oligomorphically on any one orbit.  In particular, if $G$ is a non-elementary hyperbolic group acting faithfully on $\partial_h G$ we obtain a faithful oligomorphic action of $[[G]]$ on a certain countable dense subset of $\partial_h G$. 
 Using an argument of the first author, James Hyde, and Francesco Matucci \cite{Belk-Hyde-Matucci-2}, the stabilizers of finite sets in this action are finitely generated, so it follows from the above theorem that $[[G]]$ embeds into a finitely presented simple group.

\section{Acknowledgements}
The authors would like to thank Carl-Fredrik Nyberg Brodda for his translation of Kuznetsov's article \cite{Kuznetsov}.  We would also like to thank Francesco Matucci and Matthew Zaremsky for valuable conversations around the topic of this article. 

\bibliographystyle{plain}
\bibliography{biblio.bib}
 \end{document}